\documentclass[12pt]{amsart}
\UseRawInputEncoding
\usepackage{color}
\usepackage{dsfont}
\usepackage{url}
\usepackage{hyperref}
\usepackage{mathtools}
\usepackage{anysize}
\usepackage{geometry}
\marginsize{2cm}{2cm}{2cm}{2cm}
\usepackage[normalem]{ulem}
\theoremstyle{plain}
\numberwithin{equation}{section}
\newtheorem{thm}{Theorem}[section]

\newtheorem{lem}[thm]{Lemma}

\newcommand{\PP}[1]{\left(#1\right)}
\newcommand{\CC}[1]{\left[#1\right]}
\newcommand{\LL}[1]{\left\{#1\right\}}

\newcommand{\abs}[1]{\left|#1\right|}
\renewcommand{\Pr}[1]{\mathbb{P}\left(#1\right)}

\newcommand{\R}{\mathbb{R}}
\newcommand{\C}{\mathbb{C}}
\newcommand{\N}{\mathbb{N}}

\newcommand{\cA}{\mathcal{A}}

\title[Condition Number of Random Tridiagonal Toeplitz Matrix]{Condition Number of Random Tridiagonal Toeplitz Matrix}
\author{Paulo Manrique-Mir\'on}
\email{manriquemiron@gmail.com}
\begin{document}
\begin{abstract}
In this manuscript it is considered the eigenvalues $\lambda_j$ of a random tridiagonal Toeplitz matrix $T$. We study the asymptotic behavior of the joint distribution of $\PP{\abs{\lambda}_{\min} ,\abs{\lambda}_{\max}}$. From this, we obtain the asymptotic distribution of the condition number when $T$ is symmetric. In the non-symmetric case, we understand well the singularity of the matrix and can give some good estimation of its condition number. It is remarkable that in both these cases, it is only necessary to consider two or three random variables, but this simplicity is apparent since the structure of the tridiagonal Toeplitz matrix provides non-trivial relation between them, which also induce the asymptotic behavior is completely determined by these input random variables. Also, we want to remark that our results are satisfied under mild conditions on the random variables.
\end{abstract}
\maketitle

\setcounter{section}{1}

\subsection{\textbf{Introduction}}\label{sec:Intro}

\markboth{{Condition Number of Random Tridiagonal Toeplitz Matrix}}{{Asymptotic Distribution of Condition Number of Random Tridiagonal Toeplitz Matrix}}

Tridiagonal Toeplitz matrix is a very usual and useful structured matrix which appears in many situations as theoretical and numerical problems, for example see \cite{noschese2013tridiagonal} and the references in there. Thanks to its structure it is possible to obtain explicit formulas for quantities of interest as eigenvalues, eigenvectors, or determinant. In the case that the matrix is symmetric, it is possible to give an explicit expression for its condition number. The condition number of a matrix is a useful quantity to understand the numerical stability of an algorithm when it uses the matrix \cite{smale1985efficiency}, \cite{wozniakowski1976numerical}. In fact, this manuscript is dedicated to analyze the condition number of a random tridiagonal Toeplitz matrix. 

A tridiagonal Toeplitz matrix is a matrix which has the following form

\[
\left[\begin{array}{ccccccc}
\delta & \tau &  & & & &  \\
\sigma & \delta & \tau & & & 0 & \\
 & \cdot & \cdot & \cdot & & &\\
& & \cdot & \cdot & \cdot & &\\
&  &  & \cdot &\cdot & \cdot &\\
& 0 &  &  & \sigma & \delta & \tau \\
& & &  & & \sigma & \delta \\
\end{array}\right]
\]
where $\delta,\tau,\sigma\in\C$. When $\delta = X$, $\tau = Y$, and $\tau = Z$, where $X,Y,Z$ are random variables, we say the tridiagonal Toeplitz matrix is random. 

Let $\cA\in \mathbb{C}^{n\times m}$ be a matrix of dimension $n\times m$.
We denote the singular values of $\cA$ in non-decreasing order by $0\leq \sigma^{(n,m)}_{1}(\cA)\leq \cdots \leq \sigma^{(n,m)}_{n}(\cA)$. 
That is to say, they are the {square roots of the } eigenvalues of the $n$-square matrix $\cA^*\cA$, where $\cA^*$ denotes the conjugate transpose matrix of $\cA$.
 The condition number of $\cA$, $\kappa(\cA)$, is defined as
\begin{equation}\label{eq:cnA}
\kappa(\cA):=\frac{\sigma^{(n,m)}_{n}(\cA)}{ \sigma^{(n,m)}_{1}(\cA)}\quad \textrm{ whenever }\quad \sigma^{(n,m)}_{1}(\cA)>0.
\end{equation}

Usually, to compute the condition number is hard. When it is considered a random matrix with all its entries random variables (rv) which are independent and Gaussian then the matrix is invertible with probability one and the limit distribution of its condition number is good understanding \cite{anderson2010exact}, \cite{chen2005condition}, \cite{shakil2018note}. In this manuscript we study the condition number when $X,Y,Z$ are real rv. 

Two remarkable aspects of our result are that a random tridiagonal Toeplitz matrix can be singular with positive probability for either discrete or continuous random entries and it is not necessary to impose strong conditions on $X,Y,Z$.

This manuscript is organized as follows. Sections~\ref{sec:MainResult} gives the statement of our main result. Section ~\ref{sec:Proof} we develop the proof of our main result and Section ~\ref{sec:Examples} we give some examples of random tridiagonal Toeplitz matrices with Radamecher, Cauchy, and Gaussian random entries.

Before starting, we use $\abs{\cdot}$as the norm of a complex number or the absolute value of a real number depending on the context.

\subsection{\textbf{Main Result}} \label{sec:MainResult}

\begin{thm} \label{thm:Jue04May1451}
Let $\mathcal{T}_n^{(3)}$ be a random tridiagonal Toeplitz matrix of order $n$ associated with the rv $X,Y, Z$. For $x,y\in\R$, the joint distribution of singular values $\sigma_{\min}\PP{\mathcal{T}_n^{(3)}}$, $\sigma_{\max}\PP{\mathcal{T}_n^{(3)}}$ satisfies:
\begin{align*}
& \lim_{n\to\infty} \Pr{\sigma_{\min}\PP{\mathcal{T}_n^{(3)}} \leq x, \sigma_{\max}\PP{\mathcal{T}_n^{(3)}}} = \\
&\,\,\,  \Pr{ M \leq y} - \Pr{x<0 \leq M \leq y, -\frac{X}{\abs{Y}}\in[-2,2]} - \Pr{x<m \leq M \leq y, -\frac{X}{\abs{Y}}\not\in[-2,2]},
\end{align*}
where $M=\max\LL{\abs{X-2\abs{Y}},\abs{X+2\abs{Y}}}$, $m=\min\LL{\abs{X-2\abs{Y}},\abs{X+2\abs{Y}}}$.

Let  $T_n^{(3)}$ be a random (non-symmetric) tridiagonal Toeplitz matrix of order $n$ associated  with the rv $X,Y,Z$. For $x,y\in\R$, the joint distribution of norm of the eigenvalues $\abs{\lambda}_{\min}\PP{T_n^{(3)}}$, $\abs{\lambda}_{\max}\PP{T_n^{(3)}}$ satisfies:
\begin{align*}
& \lim_{n\to\infty}\Pr{\abs{\lambda}_{\min}\PP{T_n^{(3)}} \leq x,  \abs{\lambda}_{\max}\PP{T_n^{(3)}} \leq y} =\\
&\,\,\, \Pr{\textnormal{M}\leq y} -\Pr{YZ=0, x<\abs{X}\leq y} - \Pr{YZ>0, -\frac{X}{2\sqrt{YZ}}\in[-1,1], x<0\leq \textnormal{M} \leq y}  \nonumber\\
& \, - \Pr{YZ>0, -\frac{X}{2\sqrt{YZ}}\not\in[-1,1], x<\mu\leq \textnormal{M} \leq y} - \Pr{YZ<0, x<\abs{X}\leq \textnormal{M}\leq y}, 
\end{align*}
where $\mu := \min\LL{\abs{X-2\sqrt{YZ}},\abs{X+2\sqrt{YZ}}}$, $\textnormal{M}:= \max\LL{\abs{X-2\sqrt{YZ}},\abs{X+2\sqrt{YZ}}}$.
\end{thm}

Note that the first part of Theorem~\ref{thm:Jue04May1451} is a consequence of the second part, taking $Z$ as $Y$. But the proof of this second part is based in the symmetric case, where it is possible to talk directly on singular values, whereas in the non-symmetric case we can only consider the eigenvalues. However, if $A$ is $n\times n$ matrix, then we can see directly that
\[
\sigma_{\min}(A)\leq \abs{\lambda_j(A)} \leq \sigma_{\max}(A), \;\;\; j=1,\ldots, n,
\]
where $\lambda_j(A)$ is an eigenvalue of $A$. Thus, the second part gives of lower bound of the distribution of condition number of $T_n^{(3)}$, i.e.,
\begin{equation}\label{eq:Jue04May1539}
\Pr{\frac{\abs{\lambda}_{\max}\PP{T_n^{(3)}}}{\abs{\lambda}_{\min}\PP{T_n^{(3)}}} \geq z} \leq \Pr{\frac{\sigma_{\max}\PP{T_n^{(3)}} }{ \sigma_{\min}\PP{T_n^{(3)}} } \geq z},
\end{equation} which is useful when the quotient inside of the left side is larger with positive probability. Moreover, the singularity of $T_n^{(3)}$ is good understanding due to
\[
\Pr{\sigma_{\min}\PP{T_n^{(3)}} = 0} = \Pr{ \abs{\lambda}_{\min}\PP{T_n^{(3)}} = 0}.
\]
In Section~\ref{sec:Examples}, we can observe that $T_n^{(3)}$ has a positive probability of being singular when $n$ is sufficiently larger for some discrete or continuous distribution, which is different from the case when a random Toeplitz matrix with its all diagonal being continuous and independent rv is invertible with probability one. In the case that $Y$ and $Z$ are discrete rv with $\abs{Y}=\abs{Z}$ as in the case they follow a Rademacher distribution, $T_n^{(3)}$ is a normal matrix \cite[Theorem 3.1]{noschese2013tridiagonal} and ~\ref{eq:Jue04May1539} is an equality.

Also, it is necessary to notice that the random variables in Theorem~\ref{thm:Jue04May1451} do not have any strong restriction on their joint distribution, i.e., they can be dependent rv or do not have any moment. Finally, the structure of the tridiagonal Toeplitz matrix induces that the asymptotic behavior of the condition number depends entirely on the input random variables.

\subsection{\textbf{Proof}} \label{sec:Proof}

We start with a useful lemma, which is the main tool to prove our result.

\begin{lem}\label{lem1427} Let $X,Y$ be two non-degenerated random variables and we define from these $m:= \min\LL{\abs{X-Y},\abs{X+Y}}$ and $M:=\max\LL{\abs{X-Y},\abs{X+Y}}$. For $x,y\in\R$,
\begin{align*}
& \Pr{x<\abs{X+\alpha Y}\leq y \mbox{ for all } \alpha\in[-1,1]} \\
& \, = \Pr{x< 0\leq M \leq y, -\frac{X}{Y}\in[-1,1]} + \Pr{x<m \leq M \leq y, -\frac{X}{Y}\not\in[-1,1]}.
\end{align*}
\end{lem}
\begin{proof}
Let $f(\alpha) = X+ \alpha Y$ for $\alpha\in[-1,1]$. Note that $f(\alpha)$ is a line whose only zero is at $-X/Y$, assuming that $Y\neq 0$. If $\mathcal{A}:=\{x<\abs{X+\alpha Y}\leq y \mbox{ for all } \alpha\in[-1,1]\}$, we have
\begin{align*}
\Pr{\mathcal{A}}  & = \Pr{\mathcal{A},-\frac{X}{Y}\in[-1,1]} + \Pr{\mathcal{A},-\frac{X}{Y}\not\in[-1,1]} \\
& = \Pr{x< 0\leq M \leq y, -\frac{X}{Y}\in[-1,1]} + \Pr{x<m \leq M \leq y, -\frac{X}{Y}\not\in[-1,1]}.
\end{align*}
\end{proof}

Let $A_n=\LL{\alpha_1,\ldots,\alpha_n}\subset[-1,1]$ for each $n\in\N$ such that $\lim_{n\to\infty} A_n$ is a dense set of $[-1,1]$ with $d_n:=\max\LL{\abs{\alpha_{j}-\alpha_{j}} : \mbox{ for all } i\neq j}$ such that $\LL{d_n}$ is a decreasing sequence. For $x,y\in\R$ and $n\in\N$, we define the event $\mathcal{A}_n$ as
\[
\mathcal{A}_n :=\LL{x<\abs{X+\alpha Y}\leq y \mbox{ for all }\alpha\in A_n}.
\]
Considering the event $\mathcal{A}$ as in the proof of Lemma \ref{lem1427}. Note the event $\mathcal{A}$ implies the event $\mathcal{A}_n$ for all $n\in\N$. From this observation, we have
\[
\abs{\Pr{\mathcal{A}} - \Pr{\mathcal{A}_n}} \leq \Pr{\mathcal{A}_n\setminus \mathcal{A}}.
\]
The event $\mathcal{A}_n\setminus \mathcal{A}$ implies that two possible scenarios: 
\begin{itemize}
\item $\mathcal{E}^{(1)}_n$: there exists an interval $I_n\subset[-1,1]$ such that there is $\beta_n^*\in I_n$  with $\abs{X+\beta_n^* Y} \leq x$. Observe that $\abs{f(\alpha)} = \abs{X+ \alpha Y}$ for $\alpha\in[-1,1]$ takes its absolute minimum value at $-1$, $1$, or $-X/Y$, then $I_n$ should contain one of them. Additionally, we have that one extreme of $I_n$ should be open, its length $\abs{I_n}$ goes to zero as $n\to\infty$, and $I_{n+1}\subset I_n$ for all $n\in\N$.
\item $\mathcal{E}^{(2)}_n$: there exists interval $J_n\subset[-1,1]$ such that there is $\beta_n^{**}\in J_n$ satisfies $\abs{X+\beta_n^{**}Y} > y$. As $\abs{f(\alpha)}$ takes its absolute maximum value at $-1$ or $1$, $J_n$ should contain one of them. Additionally, one extreme of $J_n$ is open, $\abs{J_n}$ goes to zero as $n\to\infty$, and $J_{n+1}\subset J_n$ for all $n\in\N$.
\end{itemize}
From the above observations, we have
\begin{align*}
\lim_{n\to\infty} \abs{\Pr{\mathcal{A}} - \Pr{\mathcal{A}_n}} & \leq \lim_{n\to\infty} \Pr{\mathcal{A}_n\setminus \mathcal{A}} \\
& \leq \lim_{n\to\infty} \PP{\Pr{\mathcal{E}^{(1)}_n} + \Pr{\mathcal{E}^{(2)}_n}} = 0.
\end{align*}
Thus, we can conclude that 
\begin{equation}\label{eqnJue20abr1958}
\lim_{n\to\infty} \Pr{\mathcal{A}_n} = \Pr{\mathcal{A}}.
\end{equation}\hfill$\Box\Box$

The singular values of a random symmetric tridiagonal Toeplitz matrix $\mathcal{T}_n^{(3)}$ of order $n$ are $\abs{X+2\abs{Y}\cos\PP{\frac{\pi}{n+1}j}}$ for $j=1,2,\ldots,n$, see \cite[Section 2]{noschese2013tridiagonal}. Thus
\begin{align*}
& \sigma_{\min}\PP{\mathcal{T}_n^{(3)}} = \min_{j=1,2,\ldots,n}\LL{\abs{X+2\abs{Y}\cos\PP{\frac{\pi}{n+1}j}}}, \\
& \sigma_{\max}\PP{\mathcal{T}_n^{(3)}} = \max_{j=1,2,\ldots,n}\LL{\abs{X+2\abs{Y}\cos\PP{\frac{\pi}{n+1}j}}}.
\end{align*}
Note $\LL{\cos\PP{\frac{\pi}{n+1}j} : j=1,2,\ldots,n}$ approaches to a dense set in $[-1,1]$ as $n$ goes to infinity. Taking $\alpha_j = \cos\PP{\frac{\pi}{n+1}j}$ for $j=1,2,\ldots,n$, and $x,y\in\R$, we have from the previous discussion 
\begin{align*}
& \Pr{\sigma_{\min}\PP{\mathcal{T}_n^{(3)}} \leq x, \sigma_{\max}\PP{\mathcal{T}_n^{(3)}} \leq y} =\\
 & \,\,\,\,\,\, \Pr{\abs{X+\alpha_j (2\abs{Y})} \leq y \mbox{ for all } j=1,\ldots,n} - \Pr{x<\abs{X+\alpha_j (2\abs{Y})} \leq y \mbox{ for all } j=1,\ldots,n} \\
& \,\,\,\longrightarrow \\
& \,\,\,\,\,\, \Pr{\abs{X+\alpha (2\abs{Y})}\leq y \mbox{ for all } \alpha\in[-1,1]}  - \Pr{x<\abs{X+\alpha (2\abs{Y})}\leq y \mbox{ for all } \alpha\in[-1,1]}
\end{align*}
as $n$ goest to infinity. Thus, from Lemma \ref{lem1427} is obtained
\begin{align*}
& \lim_{n\to\infty}  \Pr{\sigma_{\min}\PP{\mathcal{T}_n^{(3)}} \leq x, \sigma_{\max}\PP{\mathcal{T}_n^{(3)}} \leq y} = \\
& \;\;\,\,\, \Pr{ M \leq y} - \Pr{x<0 \leq M \leq y, -\frac{X}{\abs{Y}}\in[-2,2]} - \Pr{x<m \leq M \leq y, -\frac{X}{\abs{Y}}\not\in[-2,2]},
\end{align*}
where $M=\max\LL{\abs{X-2\abs{Y}},\abs{X+2\abs{Y}}}$, $m=\min\LL{\abs{X-2\abs{Y}},\abs{X+2\abs{Y}}}$.

By Continuous Mapping Theorem, the asymptotic distribution of $\kappa_n = \frac{\sigma_{\max}\PP{\mathcal{T}_n^{(3)}}}{\sigma_{\min}\PP{\mathcal{T}_n^{(3)}}}$ converges to $\kappa = \frac{V}{W}$, where $(W,V)$ is distributed as
\begin{align}
& \Pr{W\leq x, V\leq y} =  \\ 
& \,\,\,\,\,\, \Pr{ M \leq y} - \Pr{x<0 \leq M \leq y, -\frac{X}{\abs{Y}}\in[-2,2]} - \Pr{x<m \leq M \leq y, -\frac{X}{\abs{Y}}\not\in[-2,2]},  \nonumber
\end{align}
and from this, we follow that
\begin{align}
\Pr{V\leq y} & = \Pr{M\leq y}, \label{eq:mar25abr1150}\\
\Pr{W\leq x} & = 1 - \Pr{x<0,-\frac{X}{\abs{Y}}\in[-2,2]} - \Pr{x<m,-\frac{X}{\abs{Y}}\not\in[-2,2]}.
\end{align}\hfill$\Box\Box$\\

The eigenvalues of a random (non-symmetric) tridiagonal Toeplitz matrix $T_n^{(3)}$ of order $n$ are $\lambda_j = X+2\sqrt{YZ}\cos\PP{\frac{\pi j}{n+1}}$ for $j=1,2,\ldots, n$, see \cite[Section 2]{noschese2013tridiagonal}. Note that $\lambda_j$ can be complex number. Thus
\begin{align*}
& \abs{\lambda}_{\min}\PP{T_n^{(3)}} = \min_{j=1,2,\ldots,n}\LL{\abs{X+2\sqrt{YZ}\cos\PP{\frac{\pi}{n+1}j}}}, \\
& \abs{\lambda}_{\max}\PP{T_n^{(3)}} = \max_{j=1,2,\ldots,n}\LL{\abs{X+2\sqrt{YZ}\cos\PP{\frac{\pi}{n+1}j}}}.
\end{align*} In the following, we assume that $X,Y,Z\in\R$ with probability 1. Using similar arguments from the symmetric case, we follow that
\begin{align*}
&  \lim_{n\to\infty} \Pr{\abs{\lambda}_{\min}\PP{T_n^{(3)}} \leq x,  \abs{\lambda}_{\max}\PP{T_n^{(3)}} \leq y} = \Pr{\mathcal{W} \leq x, \mathcal{V} \leq y},
\end{align*}
where $\PP{\mathcal{W},\mathcal{V}}$ is distributed as
\begin{align}
&\Pr{\mathcal{W} \leq x, \mathcal{V} \leq y}=\\
&\,\,\, \Pr{\mbox{M}\leq y} -\Pr{YZ=0, x<\abs{X}\leq y} - \Pr{YZ>0, -\frac{X}{2\sqrt{YZ}}\in[-1,1], x<0\leq \mbox{M} \leq y}  \nonumber\\
& \, - \Pr{YZ>0, -\frac{X}{2\sqrt{YZ}}\not\in[-1,1], x<\mu\leq \mbox{M} \leq y} - \Pr{YZ<0, x<\abs{X}\leq\mbox{M}\leq y}, \nonumber
\end{align} where $\mu := \min\LL{\abs{X-2\sqrt{YZ}},\abs{X+2\sqrt{YZ}}}$, $\mbox{M}:= \max\LL{\abs{X-2\sqrt{YZ}},\abs{X+2\sqrt{YZ}}}$. Additionally, 
\begin{align}
\Pr{\mathcal{V}\leq y} & = \Pr{\mbox{M}\leq y}, \label{eq:mar25abr1150}\\
\Pr{\mathcal{W}\leq x} & = 1- \Pr{YZ=0, x<\abs{X}} - \Pr{YZ>0, -\frac{X}{2\sqrt{YZ}}\in[-1,1], x<0} \\
& \,\,\, \,\,\, - \Pr{YZ>0, -\frac{X}{2\sqrt{YZ}}\not\in[-1,1], x<\mu} - \Pr{YZ<0, x<\abs{X}},
\end{align}\hfill$\Box\Box$

\subsection{\textbf{Examples}} \label{sec:Examples}
In the following we give some particular cases of our main result. We assume that the random variables $X,Y,Z$ are independent.\\

\begin{enumerate}
\item $X,Y$ have Rademacher distribution. In the symmetric case, by a direct computation is obtained $m=1$ and $M=3$ with probability $1$, and $\lim_{n\to\infty} \Pr{\sigma_{\min}\PP{\mathcal{T}_n^{(3)}} = 0} = 1$, i.e., $\kappa = +\infty$ with probability 1. In other words, a random triangular symmetric Toeplitz matrix is ill-conditioning with high probability when $n$ is larger.\\

\item $X,Y$ have Standard Cauchy distribution. In the symmetric case, we have 
\begin{align}
\Pr{ V \leq y} & = \Pr{ M \leq y}  = 2 \int_{0}^{\infty} \Pr{\abs{X-2v}\leq y,\abs{X+2v}\leq y}\frac{1}{\pi(1+v^2)} \mbox{d}v \nonumber\\
& = \frac{4}{\pi^2} \mathds{1}_{y\geq 0} \int_{0}^{y/2} \int_{0}^{y-2v}  \frac{1}{(1+t^2)(1+v^2)} \mbox{d}t \mbox{d}v.
\end{align}
Note
\begin{equation}
c_2 := \Pr{W=0} = \frac{4}{\pi^2} \int_{0}^{\infty} \frac{\arctan(2v)}{1+v^2} \mbox{d}v \approx 0.636834,
\end{equation}
i.e., a random tridiagonal Toeplitz matrix with Standard Cauchy distribution is singular with high probability when its dimension is larger. By Leibniz integral rule we have that density of $W\mathds{1}_{W>0}$ is 
\begin{equation}
f_{W\mathds{1}_{W>0}}(w) = \frac{4}{\pi^2(1-c_2)} \mathds{1}_{w>0} \int_{0}^{\infty}\frac{1}{(1+v^2)(1+(w+2v)^2)} \mbox{d}v
\end{equation}
Now, for $z\geq 0$, the distribution of $\kappa$ when the matrix is not singular is
\begin{align*}
& \Pr{\frac{V}{W\mathds{1}_{W>0}}\leq z}  = \int_{-\infty}^{\infty} \Pr{V\leq zw | W\mathds{1}_{W>0}= w} f_{W\mathds{1}_{W>0}}(w) dw \\
&\;\;\; = \frac{16}{\pi^4(1-c_2)}\int_{0}^{\infty} \CC{ \int_{0}^{zw/2} \int_{0}^{zw-2v}  \frac{1}{(1+t^2)(1+u^2)} \mbox{d}t \mbox{d}u} \CC{\int_{0}^{\infty}\frac{1}{(1+v^2)(1+(w+2v)^2)} \mbox{d}v} \mbox{d}w.
\end{align*}

\item $X, Y$ have Standard Normal distribution. In the symmetric case we have the following. Let $\Phi(t)$ and $\phi(t)$ be the distribution and density, respectively, of Standard Normal rv. Then
\begin{align}\label{eq:lun24abr929}
\Pr{ V \leq y} & = \Pr{ M \leq y} = 2 \int_{0}^{\infty} \Pr{\abs{X-2v}\leq y,\abs{X+2v}\leq y}\phi(v) 
\mbox{d}v \nonumber\\
& = 2\mathds{1}_{y\geq 0} \int_{0}^{\infty} \mathds{1}_{y\geq 2v} \CC{\Phi(y-2v) - \Phi(-y+2v)} \phi(v) 
\mbox{d}v \\
& = 2\mathds{1}_{y\geq 0} \int_{0}^{y/2} \CC{\Phi(y-2v) - \Phi(-y+2v)} \phi(v) \nonumber
\mbox{d}v.
\end{align}

Note 
\begin{equation}\label{eq:26abr20231416}
c_0 := \Pr{W=0}=4\int_{0}^{\infty} \CC{\Phi(2v)-\Phi(0)}\phi(v)\mbox{d}v= \frac{\pi + \tan^{-1}\PP{\frac{24}{7}}}{2\pi} \approx 0.704832,
\end{equation}
 i.e., $\kappa = +\infty$ with probability bigger than $0.7$, in other words, a random tridiagonal symmetric Toeplitz matrix with Standard Normal distribution is ill-conditioning with high probability. By Leibniz integral rule we have that density of $W\mathds{1}_{W>0}$ is
\begin{equation}\label{eqn:lun24abr1539}
f_{W\mathds{1}_{W>0}}(w)=\frac{4}{1-c_0} \mathds{1}_{ w > 0}\int_{0}^{\infty} \phi(w+2v)\phi(v) \mbox{d}v.
\end{equation}
Now, for $z\geq 0$, the distribution of $\kappa$ when it is finite is

\begin{align*}
& \hspace{-0.5cm}\Pr{\frac{V}{W\mathds{1}_{W>0}}\leq z}  = \int_{-\infty}^{\infty} \Pr{V\leq zw | W\mathds{1}_{W>0}= w} f_{W\mathds{1}_{W>0}}(w) dw \\
& = \int_{-\infty}^{\infty} \CC{2\mathds{1}_{zw\geq 0} \int_{0}^{zw/2} \CC{\Phi(zw-2t) - \Phi(-zw+2t)} \phi(t)\mbox{d}t}  \CC{\frac{4}{1-c_0} \mathds{1}_{ w > 0}\int_{0}^{\infty} \phi(w+2v)\phi(v) \mbox{d}v} dw \\
& = \frac{8}{1-c_0}\int_{0}^{\infty}  \CC{\int_{0}^{zw/2} \CC{\Phi(zw-2t) - \Phi(-zw+2t)} \phi(t)\mbox{d}t} \CC{\int_{0}^{\infty} \phi(w+v)\phi(v) dv} dw.
\end{align*}\\

\item $X,Y,Z$ have Rademacher distribution. In the non-symmetric case, we have that $\mu\in\LL{1,\sqrt{5}}$ and $\mbox{M}\in\LL{3,\sqrt{5}}$ with probability 1. Moreover,
\begin{align*}
\Pr{\mathcal{W}=0} & = 1 - \Pr{YZ>0, \frac{X}{\sqrt{YZ}}\not\in[-2,2]} - \Pr{YZ < 0} = \frac{1}{2},
\end{align*}
i.e., a random non-symmetric tridiagonal Toeplitz matrix $T_n^{(3)}$ of order $n$ is singular with probability approximately $0.5$ when $n$ is larger, in other words, it is ill-conditionally half the time.\\

\item $X, Y, Z$ have Standard Normal distribution. In the non-symmetric case we have the following. The distribution of $\mathcal{V}$ is
\begin{align}\label{eqn:lun01may1257}
\Pr{\mathcal{V}\leq y} & = \Pr{\mbox{M}\leq y} = 2\int_{0}^{\infty} \Pr{\abs{v-2\sqrt{YZ}}\leq y,\abs{v+2\sqrt{YZ}}\leq y} \phi(v) \mbox{d}v \\
& = 2 \mathds{1}_{y\geq 0} \int_{0}^{y} \CC{\int_{0}^{\PP{\frac{y-v}{2}}^2} f_{YZ}(t) \mbox{d}t + \int_{0}^{(y^2-v^2)/4} f_{YZ}(t)\mbox{d}t} \phi(v)\mbox{d}v, \nonumber
\end{align}
where 
\begin{equation}
f_{YZ}(t) = \frac{1}{\pi} \int_{0}^{\infty}\frac{1}{s}\exp\PP{-\frac{1}{2}\PP{s^2+\frac{t^2}{s^2}}} \mbox{d}s.
\end{equation}
Note
\begin{equation}\label{eq:lun1may1131}
c_1 := \Pr{W=0}= 2\int_{0}^{\infty} \Phi(2\sqrt{t}) f_{YZ}(t) \mbox{d}t -\frac{1}{2} \approx 0.351488.
\end{equation}
i.e, a random tridiagonal matrix with Standard Normal distribution is ill-conditioning with probability approximately $0.351488$ when its order is larger. By Leibniz integral rule we have that density of $\mathcal{W}\mathds{1}_{\mathcal{W}>0}$ is
\begin{equation}\label{eqn:lun01may1247}
f_{\mathcal{W}\mathds{1}_{\mathcal{W}>0}}(w)=\frac{1}{1-c_1} \mathds{1}_{w>0}\CC{\phi(w) + 2\int_{0}^{\infty} \phi(w+2\sqrt{t})f_{YZ}(t)\mbox{d}t}.
\end{equation}
Now, for $z\geq 0$, the distribution of $\mathcal{V}/\mathcal{W}$ when it is finite is
\begin{align*}
& \Pr{\frac{\mathcal{V}}{\mathcal{W}\mathds{1}_{\mathcal{W}>0}}\leq z}  = \int_{-\infty}^{\infty} \Pr{\mathcal{V}\leq zw | \mathcal{W}\mathds{1}_{\mathcal{W}>0}= w} f_{\mathcal{W}\mathds{1}_{\mathcal{W}>0}}(w) \mbox{d}w,
\end{align*}
\end{enumerate}
where it is needed to replace the expressions ~\ref{eqn:lun01may1257}, ~\ref{eqn:lun01may1247} appropriately. 


\end{document}